\documentclass[oneside, a4paper,11pt,reqno]{amsart}
\usepackage{bbm}
\usepackage[T1]{fontenc}
\usepackage{amssymb}
\usepackage{mathpazo}
\usepackage{verbatim}
\usepackage{amsmath}
\usepackage{enumitem}

\usepackage{amsfonts}
\usepackage{color}
\definecolor{Red}{rgb}{1,0,0}

\newtheorem{theorem}{Theorem}
\newtheorem{proposition}[theorem]{Proposition}
\newtheorem{cor}[theorem]{Corollary}
\newtheorem{lemma}[theorem]{Lemma}
\newtheorem{obs}[theorem]{Remark}

\newtheorem{assump}[theorem]{Assumptions}

\def \E {\mathbb{E}}
\def \PP {\mathbb{P}}
\def \N  {\mathbb{N}}
\def \R  {\mathbb{R}}

\def \Z  {\mathbb{Z}}

\begin{document}
\title[New examples of ballistic RWRE]{New examples of ballistic RWRE in the low disorder regime}
\author{Alejandro F. Ram\'\i rez$^*$ and Santiago Saglietti$^\dagger$}

\email{($*$) aramirez@mat.uc.cl ($\dagger$) sjs30@nyu.edu}
\address{($*$) 	Facultad de Matem\'aticas, Pontificia Universidad Cat\'olica de Chile\newline
	($\dagger$) Faculty of Industrial Engineering and Management, Technion, Israel and NYU-ECNU Institute of Mathematical Sciences at NYU Shanghai.}

\thanks{A. F. Ram\'\i rez has been
	partially supported by Fondo Nacional de Desarrollo Cient\'ifico y Tecnol\'ogico 1141094 and
	1180259, and Iniciativa Cient\'ifica Milenio. \\
	S. Saglietti has been supported in part at the Technion by a fellowship from the Lady Davis Foundation, the Israeli Science Foundation (ISF) grants no. 1723/14 and 765/18, and by the NYU-ECNU Institute of Mathematical Sciences at NYU Shanghai.}

\keywords{Random walk in random environment, small perturbations of simple random walk, ballistic behavior, concentration inequalities.}
\subjclass[2010]{60K37, 82D30, 82C41.}

\maketitle

\begin{abstract}
	We give a new criterion for ballistic behavior of random walks in random environments which are low disorder perturbations of the simple symmetric random walk on $\Z^d$, for $d\geq 2$. This extends the results from 2003 established by Sznitman in \cite{Sz03} and, in particular, allow us to give new examples of ballistic RWREs in dimension $d=3$ which do not satisfy Kalikow's condition,
	through a new sharp version of Kalikow's criteria. Essentially, this new criterion states that ballisticity occurs whenever the average local drift of the walk is not too small when compared to the standard deviation of the environment. Its proof relies on applying coarse-graining methods together with a variation of the Azuma-Hoeffding concentration inequality in order to verify the fulfillment of a ballisticity condition by Berger, Drewitz and Ram\'irez. 
\end{abstract}

\section{Introduction and Main Results}

\subsection{Introduction} 
The random walk in a random environment
is one of the fundamental models describing the movement of a particle
in disordered media (see \cite{Sz06,Z12} for a comprehensive overview of the model). For walks on $\Z^d$ with $d\ge 2$, few results exist giving explicit
formulas  for basic associated quantities such as the velocity, asymptotic direction or variance, or conditions characterizing specific long-term behavior such as transience/recurrence,  directional transience and ballistic movement.
In particular, it is still a widely open problem to explicitly characterize the (law of the) small disorder necessary to produce ballistic behavior whenever added to the jump probabilities of the simple symmetric random walk, see \cite{Sz03,Sa04}. In this article we focus on this particular question and generalize previously known conditions by exploring the use
of refined concentration inequalities which are variations of the well-known Azuma-Hoeffding inequality (see \cite{BGL13} for an extensive review of general concentration inequalities and their applications). 

In the case of an i.i.d. random environment in dimension $d \geq 3$, Sznitman was able to derive in \cite{Sz03} conditions on the small disorder which guarantee that the perturbed walk is ballistic. Essentially, he showed that as long as the average local drift in some direction $\ell$ of the perturbed random walk is not \textit{too} small with respect to $\epsilon$, the $L^\infty$-norm of the perturbation, one has ballisticity in direction $\ell$ (see Section \ref{sec:main} below for a precise statement). In the present article we improve on this by showing that, in fact, one only needs the average drift to be not too small but compared instead to $\sigma$, the standard deviation of the environment, which is always a smaller quantity than $\epsilon$ (and, furthermore, could potentially be \textit{much} smaller).

As a consequence of this improvement we are able to obtain new examples of RWREs with ballistic behavior. As a matter of fact, to construct these examples it is necessary to show that they satisfy certain
		conditions generally known as {\it ballisticity conditions}. Two important ballisticity conditions are
		Kalikow's condition, introduced in \cite{K81}, which is a requirement on the averaged local drift,
		and the strictly weaker (at least in dimensions $d\ge 3$) condition (T') , introduced by Sznitman in \cite{Sz02}, which is the requirement that the exit probability
		of the random walk from slabs perpendicular to the moving direction of the random walk decay as stretched exponentials.
		The examples we obtain in this article are  basically of two types:
(i) new examples in all dimensions $d \geq 2$ of (small disorder) RWREs satisfying Kalikow's condition for ballisticity; and, most importantly, (ii) new examples in dimension $d=3$ of (small disorder) RWREs satisfying the polynomial ballisticity condition from \cite{BDR14} and hence (T') but \textit{not} Kalikow's condition (the first examples of type (ii) appear in \cite{Sz03} for all dimensions $d \geq 3$). These new three-dimensional examples are of independent interest, since they show that the class of random walks which satisfy (T') but not Kalikow's condition might be much larger and complex than what one would have expected from the previous results of \cite{Sz03}.

We prove our results by refining some of the estimates on \cite{Sz03} and ideas developed in \cite{LRSS}. A novel and key ingredient in our approach is the use of fine concentration inequalities, namely a variation of the well-known Azuma-Hoeffding inequality which gives a suitable control on the size of martingale differences in terms of their conditional variance. 

Before we state our results more precisely and give further details/discussion, let us formally introduce the model and set up the framework to be used throughout the article.

\subsection{The model}

Fix an integer $d \geq 2$ and for each $x=(x_1,\dots,x_d) \in \Z^d$ let $|x|:=|x_1|+\dots+|x_d|$ denote its $\ell^1$-norm. Let $V:=\{x\in\mathbb Z^d: |x|=1\}$ be the set of canonical vectors in $\R^d$ and denote by $\mathcal P$ the set of all probability vectors $\vec{p}=(p(e))_{e \in V}$ on $V$, i.e. $\vec{p} \in [0,1]^V$ such that $\sum_{e\in V}p(e)=1$. In addition, consider the product space $\Omega:=\mathcal P^{\mathbb Z^d}$ 
with its Borel $\sigma$-algebra, denoted by $\mathcal{B}(\Omega)$. We will call any element $\omega=(\omega(x))_{x \in \Z^d} \in \Omega$ an \textit{environment} (on $\Z^d$). For each $x\in\mathbb Z^d$, $\omega(x)$ is a probability vector on $V$, whose components we denote by $\omega(x,e)$, i.e. $\omega(x)=(\omega(x,e))_{e \in V}$. 
The {\it random walk in the environment} $\omega$
starting from $x\in\Z^d$ is then defined as the Markov chain $X=(X_n)_{n \in \N_0}$ on $\Z^d$ which starts from $x$ and is given by the transition probabilities
$$
P_{x,\omega}(X_{n+1}=y+e|X_n=y)=\omega(y,e) \hspace{1cm}\text{ for each }y\in\mathbb Z^d\,,\,e \in V.
$$ 
We denote its law  by $P_{x,\omega}$.
We assume throughout that the space of environments $\Omega$ is
endowed with a probability measure $\PP$, called the \mbox{{\it environmental law}.}
We shall call $P_{x,\omega}$  the {\it quenched law} of the random walk, and also refer to the semi-direct product \mbox{$P_x:= \PP\otimes P_{x,\omega}$} on
$\Omega\times{(\mathbb Z^d)}^{\mathbb N_0}$ given by 
$$
P_x(A \times B) = \int_A P_{x,\omega}(B) d\PP(\omega)
$$
as the {\it averaged} or  {\it annealed law} of the random walk. In general, we will call the sequence $(X_n)_{n \in \N_0}$ under the annealed law a \textit{random walk in a random environment} (RWRE) with \mbox{environmental law $\PP$.} 

\begin{assump}\label{assump}
	Throughout the sequel we shall make the following assumptions on $\PP$:
	\begin{enumerate}
		\item [A1.] The family of random probability vectors $(\omega(x))_{x \in \Z^d}$  is i.i.d. with some common law $\mu$ on $\mathcal{P}$. Equivalently, $\PP$ is the product measure on $\Omega$ with marginal distribution $\mu$.
		\item [A2.] $\PP$-almost surely, every weight $\omega(x)$ is a small perturbation of the weights of the simple symmetric random walk, i.e.
		\begin{equation} \label{eq:perturb}
		\epsilon=\epsilon(\mu):=4d \left\|\omega(0) - \left(\frac{1}{2d},\dots,\frac{1}{2d}\right)\right\|_{L^\infty(\mu)} \in (0,1), 
		\end{equation} where for any random vector $\vec{p}=(p(e))_{e \in V}$ we define its $L^\infty(\mu)$-norm as
		$$
		\|\vec{p}\|_{L^\infty(\mu)}:=\inf \{ M > 0 : |p(e)| \leq M \text{ }\mu\text{-almost surely for all }e \in V\}.
		$$
	\end{enumerate}
	(The $4d$ factor in \eqref{eq:perturb} is merely a more convenient normalization for our purposes.)
\end{assump}

Observe that, by (A1), \eqref{eq:perturb} implies that there exists an event $\Omega_{\epsilon}$ with $\PP(\Omega_{\epsilon})=1$ such that on $\Omega_\epsilon$ one has that
\begin{equation}
\label{epsilon-condition}
\left|\omega(x,e)-\frac{1}{2d}\right|\le
\frac{\epsilon}{4d} \,\text{ for all }x \in \Z^d \text{ and }e \in V.
\end{equation} In particular, $\PP$ is \textit{uniformly elliptic} with ellipticity constant
\begin{equation}
\label{eq:kappa}
\kappa:=\frac{1}{4d},
\end{equation} i.e. $\PP$-almost surely one has that $\omega(x,e)\ge\kappa$ for all $x\in\mathbb Z^d$ and $e\in V$.

The goal of this article is to study transcience/ballisticity properties of $X$ on fixed directions. Recall that, given $\ell\in S^{d-1}$, one says that the random walk $X$ is \textit{transient in direction $\ell$} 
if
$$
\lim_{n\to\infty}X_n\cdot \ell=+\infty\quad P_0-a.s.,
$$
and that it is {\it ballistic in direction $\ell$} if it satisfies the stronger condition
$$
\liminf_{n\to\infty}\frac{X_n\cdot \ell}{n}>0\qquad P_0-a.s.
$$
Any RWRE which is ballistic in some direction $\ell$ satisfies also a law of large numbers (see \cite{DR14}), i.e. there exists a deterministic vector $\vec{v} \in \R^d$ with $\vec{v} \cdot \ell > 0$ such that
$$
\lim_{n \rightarrow +\infty} \frac{X_n}{n} = \vec{v} \qquad P_0-a.s..
$$ This vector $\vec{v}$ is known as the \textit{velocity} of the random walk. 

In the sequel we will fix a certain direction, let us say $e_1:=(1,0,\dots,0) \in S^{d-1}$ for example, and study transience/ballisticity only in this fixed direction. Thus, whenever we speak of transience or ballisticity of $X$ it will be understood that it is with respect to this given direction $e_1$. However, we point out that all of our results can be adapted for any other particular direction. 

\subsection{Main results}\label{sec:main}

For $x \in \Z^d$  define the {\it local drift of the RWRE at site $x$} as the random vector
$$
\vec{d}(x):=\sum_{e\in V}\omega(x,e)e.
$$ and let $\lambda$ denote the \textit{average local drift in direction $e_1$} by the formula
$$
\lambda=\lambda(\mu):= \E( \vec{d}(0) \cdot e_1 ) =  \E(\omega(0,e_1)-\omega(0,-e_1)).
$$ A natural question to ask is whether, under Assumptions \ref{assump}, $X$ is ballistic as soon as $\lambda > 0$ holds. Unfortunately, this is not the case, as one can see from \cite{BSZ03} where examples of random walks with diffusive behavior and non-vanishing $\lambda$ are constructed. However, in \cite{Sz03} it is shown that if $d \geq 3$ and $\lambda$ is not \textit{too} small with respect to $\epsilon$ then the random walk is indeed ballistic and, in fact, satisfies the so-called (T') condition for ballisticity (see Section below \ref{sec:prelim} for further details). The validity of (T') not only implies ballisticity of $X$, but also a CLT and large deviation controls for the sequence $(\frac{X_n}{n})_{n \in \N}$, see \cite{Sz02}. More precisely, in \cite{Sz03} it is shown that, under Assumptions \ref{assump}, given any $\eta \in (0,1)$ there exists some $\epsilon_0=\epsilon_0(\eta,d) \in (0,1)$ such that if $\epsilon \leq \epsilon_0$ 
and 
\begin{equation}\label{eq:sz}
\lambda \geq \begin{cases} \epsilon^{2.5-\eta} & \text{ if }d=3\\
\epsilon^{3-\eta}& \text{ if }d\geq 4\end{cases}
\end{equation} then $X$ satisfies condition (T') in direction $e_1$ and is, thus, ballistic in this direction.

Our goal in this article is to extend the results from \cite{Sz03} and provide new instances of ballistic behavior, by considering the \textit{standard deviation of the environment} $\omega$, defined as  
$$
\sigma =\sigma(\mu):= \sqrt{\sum_{e \in V} \text{Var}_\mu(\omega(0,e))} = \|\widetilde{\delta}(0)\|_{L^2(\mu)},
$$ where, for each $x \in \Z^d$, $\widetilde{\delta}(x):=(\widetilde{\delta}(x,e))_{e \in V}$ denotes the centered vector given by
\begin{equation} \label{eq:deftilde}
\widetilde{\delta}(x,e):=\omega(x,e)-\E(\omega(x,e)).
\end{equation} 

Our first result is an extension of the main result from \cite{Sz03} for the case $d=3$. 

\begin{theorem}
	\label{theorem1} Suppose that $d=3$ and Assumptions \ref{assump} hold. Then, for any $\eta\in (0,1)$ there exists $\epsilon_0=\epsilon_0(\eta)\in (0,1)$ such that if $\epsilon \leq \epsilon_0$ and 
	\begin{equation}
	\label{eq:condrift}
	\lambda \geq \sigma \epsilon^{1.5-\eta}
	\end{equation} then condition (T') is satisfied in direction $e_1$. In particular, $X$ is ballistic in direction $e_1$.
\end{theorem}

Note that, since $\sigma \leq \epsilon$ automatically holds by properties of the $L^p$ norms, \eqref{eq:condrift} is indeed weaker than the $\lambda \geq \epsilon^{2.5-\eta}$ condition stated in \cite{Sz03} for $d=3$. Furthermore, as opposed to Sznitman's original result, \eqref{eq:condrift} shows that it is possible to have ballistic behavior for environments with drift $\lambda$ as small as one wants with respect to $\epsilon$, as long as their standard deviation $\sigma$ is accordingly small. Indeed, Theorem \ref{theorem1} suggests that, in order to see ballistic behavior, what is important is that $\lambda$ is not too small but with respect to $\sigma$ rather than $\epsilon$. 

Next, we investigate the validity of Kalikow's condition for ballisticity. We refer to Section \ref{sec:prelim} for a precise definition of this condition, but for now recall the reader that it is in general strictly stronger than (T'), and that it implies slightly stronger results on the sequence $(\frac{X_n}{n})_{n \in \N_0}$, see \cite{Sz01}. We have the following result, which is valid for all dimensions $d \geq 2$.  

\begin{theorem} \label{theorem2} Suppose that $d \geq 2$ and Assumptions \ref{assump} hold. Then, if 
	$$
	\lambda > 4d\sigma^2(1 +9\epsilon),
	$$ Kalikow's condition is satisfied in direction $e_1$. In particular, $X$ is ballistic in direction $e_1$.
\end{theorem}

Note that, whenever $\sigma \ll \epsilon$, Theorem \ref{theorem2} refines \cite[Theorem 2]{LRSS}, where it was shown that Kalikow's condition holds if $\lambda > \frac{1}{d}\epsilon^2$. 

Combining both theorems together with Sznitman's result, we obtain the following corollary which yields the full map of scenarios of ballisticity known so far in this context for all $d \geq 2$. 

\begin{cor} If Assumptions \ref{assump} hold then for any $\eta \in (0,1)$ there exists $\epsilon_0 = \epsilon_0(\eta,d) \in (0,1)$ such that if $\epsilon \leq \epsilon_0$ and  
	$$
	\lambda \geq \begin{cases}
	(4d+\eta)\sigma^2 & \text{ if }d=2 \\
	\min\{ \sigma \epsilon^{1.5-\eta},(4d+\eta)\sigma^2\} & \text{ if }d=3\\
	\min\{\epsilon^{3-\eta},(4d+\eta)\sigma^2\} & \text{ if }d\geq 4
	\end{cases}
	$$ then $X$ satisfies condition (T') in direction $e_1$. In particular, $X$ is ballistic in direction $e_1$. Furthermore, if 
	$$
	\lambda \geq (4d+\eta)\sigma^2
	$$ then in fact $X$ satisfies Kalikow's stronger ballisticity condition.
\end{cor} 

Our third and last result is a variation of \cite[Theorem 5.1]{Sz03} which gives scenarios under which Kalikow's condition fails to hold. 

\begin{theorem}\label{theo:nokal} Suppose that $d \geq 2$. Then, for any given $\rho \in (0,1]$ there exists $\epsilon_0=\epsilon_0(d,\rho)> 0$ such that if the environment site distribution $\mu$ satisfies the following conditions:
	\begin{enumerate}
		\item [K1.] $\epsilon(\mu) \leq \epsilon_0$,
		\item [K2.] If $d \geq 3$ then $\mu$ is invariant under all rotations which preserve both $\Z^d$ and $e_1$ (respectively, $\mu$ is invariant under the reflection in the $e_2$-axis if $d=2$),
		\item [K3.] $\text{Var}_{\mu}(\omega(0,e_1))=\text{Var}_{\mu}(\omega(0,-e_1))$,
		\item [K4.] $\text{Var}_{\mu}(\omega(0,e_1))-\text{Cov}_{\mu}(\omega(0,e_1),\omega(0,-e_1)) \geq \rho\sigma^2>0$,
		\item [K5.] $\rho \sigma^2 > 32d^2 \lambda \geq 0$,
	\end{enumerate} then Kalikow's condition fails to hold in all directions $\ell \in S^{d-1}$. 
\end{theorem}

Theorem \ref{theo:nokal} is important mainly for two reasons. First, it shows that the condition $\lambda \geq c_1(d) \sigma^2$ from Theorem \ref{theorem2} is essentially optimal for the validity of Kalikow's condition, in the sense that there exists $c_2(d)>0$ such that if $\lambda \leq c_2(d)\sigma^2$ then one can already find examples of RWREs which do not satisfy Kalikow's condition. But, most importantly, in combination with Theorem \ref{theorem1} it also shows new examples (apart from those already given in \cite{Sz03}) of ballistic walks for $d=3$ which satisfy (T') but not Kalikow's condition. Indeed, as a direct consequence of Theorems \ref{theorem1}-\ref{theo:nokal}, we obtain the following result.

\begin{cor} Given any $\epsilon_0 > 0$ one can construct a RWRE in dimension $d=3$ verifying Assumptions \ref{assump} and such that:
	\begin{enumerate}
		\item [C1.] $\epsilon \leq \epsilon_0$ and $\lambda \leq \epsilon^{2.5}$ (so that the hypothesis \eqref{eq:sz} from \cite{Sz03} is not satisfied)
		\item [C2.] (T') is verified in direction $e_1$ but Kalikow's condition fails in all directions. 
	\end{enumerate}
\end{cor} 

One such example can be constructed as in \cite{Sz03} by first fixing $\rho \in (0,1]$ and then setting $\mu$ to be the law of the random probability weight $\omega(0)$ on $\mathcal{P}$ given by 
$$
\omega(0,e) = p(e) + \frac{\lambda}{2} e \cdot e_1 \hspace{1cm}\text{ for each $e \in V$,}
$$ for some constant $\lambda > 0$ and random probability vector $\vec{p}=(p(e))_{e \in V}$ to be specified later. It is simple to check that, if the law $\widehat{\mu}$ of $\vec{p}$ satisfies:
	\begin{enumerate}
		\item [$\bullet$] $\widehat{\mu}$ is isotropic, i.e. invariant under rotations of $\R^3$ which preserve $\Z^3$,
		\item [$\bullet$] $\text{Var}_{\widehat{\mu}}(p(e_1))-\text{Cov}_{\widehat{\mu}}(p(e_1),p(-e_1)) \geq \rho \sigma(\widehat{\mu})>0$,
	\end{enumerate} then for any $\eta \in (0,\frac{1}{2})$ one can choose constants $k_1,k_2,k_3>0$ (depending only on $\eta, \epsilon_0$ and $\rho$) in such a way if $\widehat{\mu}$ is taken to also satisfy
$$
\epsilon(\widehat{\mu}) \leq k_1 \epsilon_0 \hspace{1cm}\text{ and }\hspace{1cm}k_2 (\epsilon(\widehat{\mu}))^{1.5-\eta} \leq \sigma(\widehat{\mu}) \leq k_3 (\epsilon(\widehat{\mu}))^{1+\eta}
$$ then there exists a nonempty interval of values of $\lambda$ for which the associated walk satisfies (C1) and the hypotheses of both Theorems \ref{theorem1} and \ref{theo:nokal}, so that (C2) also holds. We omit the details.

The proof of Theorem \ref{theorem1} is an adaptation of the approach developed in \cite{Sz03}, which consists of verifying the effective criterion given in \cite{Sz02} for the validity of (T'). Instead of this criterion, we will verify that the polynomial condition from \cite{BDR14} holds, which is equivalent to (T') and more convenient for our purposes. Crucial to this verification are estimates on the average value and size of fluctuations of the Green's operator of the RWRE killed upon exiting a slab of diameter proportional to $\epsilon^{-1}$, for $\epsilon$ as in \eqref{eq:perturb}. In \cite{Sz03} it is shown that these quantities can be suitably controlled by $\epsilon$, that is, by the $L^\infty$-norm of the perturbation. However, to obtain our results we will need to show that these can still be \mbox{controlled by $\sigma$,} the $L^2$-norm of the perturbed environment, which can be, in principle, a much smaller quantity. This task requires redoing some of the estimates from \cite{Sz03} but now in $L^2$ (and thus refining them from their original $L^\infty$-form), and introducing new tools to the analysis, like the generalized version of Azuma-Hoeffding's inequality used in Lemma \ref{lemma:azuma}. 
The proof of Theorem \ref{theorem2} is inspired by that of \cite[Theorem 2]{LRSS}  and is based
on the version of Kalikow's formula proved in \cite{LRSS} together with a careful application of Kalikow's criteria for ballisticity. Finally, Theorem \ref{theo:nokal} follows from adapting the argument in \cite[Theorem 5.1]{Sz03} to the $L^2$-setting.

Finally, we mention that the ideas here can also be used to extend the bounds on the velocity developed in \cite{LRSS}. Indeed, the same approach used here goes through in \cite{LRSS} to show that:
\begin{enumerate}
	\item [B1.] In dimension $d=3$, given two quantities $\delta < \eta \in (0,1)$ there exists $\epsilon_0 \in (0,1)$ and $c_0 > 0$ depending on $\delta$ and $\eta$ such that if $\epsilon \leq \epsilon_0$ and $\lambda \geq \sigma \epsilon^{1.5-\eta}$ then $X$ satisfies (T') in direction $e_1$ and is thus ballistic in direction $e_1$ with a velocity $\vec{v}$ satisfying
	$$
	0 <	\vec{v} \cdot e_1 \leq \lambda + c_0 \sigma \epsilon^{1.5-\delta}.
	$$ 
	\item [B2.] For all dimensions $d \geq 2$, given any $\eta \in (0,1)$ there exists $\epsilon_0 = \epsilon_0(d,\eta) \in (0,1)$ such that if $\epsilon \leq \epsilon_0$ and $\lambda > (4d+\eta)\sigma^2$ then $X$ satisfies Kalikow's condition in direction $e_1$ and is this ballistic in direction $e_1$ with a velocity $\vec{v}$ satisfying 
	$$
	|\vec{v} - \lambda| \leq (4d+\eta) \sigma^2.
	$$
\end{enumerate}
However, since this extension of the results \cite{LRSS} is completely analogous to the one of \cite{Sz03} we shall do here, we will omit the details of how to prove (B1-B2). The reader interested in a proof will know how to proceed from \cite{LRSS} after reading Sections 3 and 4 below.  

The remainder of the article is organized as follows. In Section 2 we introduce  general notation and
establish a few preliminary facts about the RWRE model, including the precise definitions of Kalikow's and (T') conditions for ballisticity. Afterwards, Sections 3, 4 and 5 are each devoted to the proofs of Theorems \ref{theorem1}, \ref{theorem2} and \ref{theo:nokal}, respectively.

\section{Preliminaries}\label{sec:prelim}

In this section we introduce the general notation to be used throughout the article, as well as review some preliminary notions about RWREs we will require for the proofs.

\subsection{General notation}\label{sec:GN}
Given any subset $B\subset\mathbb Z^d$, we define
its (outer) boundary as
$$
\partial B:=\{x \in \Z^d - B : |x-y|=1 \text{ for some }y \in B\}
$$
and the first exit
time of the random walk from $B$ as
\begin{equation}\label{eq:deftb}
T_B:=\inf\{n\ge 0: X_n\notin B\}.
\end{equation} Similarly, for $x \in \Z^d$ we define the hitting time $H_x$ and first return time $\widetilde{H}_x$ of $x$ respectively as
\begin{equation}\label{eq:defhs}
H_x:=\inf\{n \geq 0 : X_n = x\}\hspace{1cm}\text{ and }\hspace{1cm}\widetilde{H}_x:=\inf \{n \geq 1 : X_n = x\}.
\end{equation} Furthermore, for $L \in \N$ we define the $L$-slab in direction $e_1$ as
\begin{equation}\label{eq:slab}
U:=\left\{y\in\mathbb Z^d: -L \leq y \cdot e_1 < L
\right\}, 
\end{equation} where we consciously omit the dependence on $L$ from the notation $U$ for simplicity. Finally, for each $M \in \N$ we also define the box
\begin{equation}
\label{beeme}
B_{M}:=\left\{y\in\mathbb Z^d: - \frac{M}{2} < y \cdot e_1 < M \text{ and }
|y\cdot e_i| < 25M^3 \text{ for } 2\le i\le d\right\}
\end{equation}
together with its {\it frontal side}
$$
\partial_+B_M:=\left\{y\in \partial B_{M}: y\cdot e_1 \geq M\right\},
$$ and its \textit{middle-frontal part}
$$
B^*_M:= \left\{ y \in B_{M} : \frac{M}{2} \leq y \cdot e_1 < M\,,\,|y \cdot e_i| < M^3 \text{ for }2 \leq i \leq d\right\}.
$$ 

\subsection{Green's functions and operators}

Let us now introduce some notation we shall use related to the Green's functions of the RWRE and of the simple symmetric random walk (SSRW). 

Given a subset $B \subseteq \Z^d$, the Green's functions of the RWRE and SSRW killed upon exiting $B$ are respectively defined for $x,y \in B \cup \partial B$ as
$$
g_B(x,y,\omega):= E_{x,\omega}\left(\sum_{n=0}^{T_{B}} \mathbbm{1}_{\{X_n=y\}}\right)\hspace{1cm}\text{ and }\hspace{1cm}g_{0,B}(x,y):=g_{B}(x,y,\omega_0),
$$ where $\omega_0$ is the corresponding weight of the SSRW, given for all $x \in \Z^d$ and $e \in V$ by
$$
\omega_0(x,e)=\frac{1}{2d}.
$$ Furthermore, if $\omega \in \Omega$ is such that $E_{x,\omega}(T_B) < +\infty$ for all $x \in B$, we can define the corresponding Green's operator on $L^\infty(B)$ by the formula
$$
G_B[f](x,\omega):= \sum_{y \in B} g_B(x,y,\omega)f(y) = E_{x,\omega}\left( \sum_{n=0}^{T_B-1} f(X_n)\right).
$$ Notice that $g_B$, and therefore also $G_B$, depends on $\omega$ only though its restriction $\omega|_B$ to $B$. Finally, it is straightforward to check that if $U$ is the slab defined in \eqref{eq:slab} then both $g_{U}$ and $G_{U}$ are well-defined for all environments $\omega \in \Omega_\epsilon$, where $\Omega_\epsilon$ is the full $\PP$-probability event under which \eqref{epsilon-condition} holds.

\subsection{Ballisticity conditions} \label{sec:bal}
We now recall the two conditions for ballisticity we shall work with: condition (T'), originally introduced by Sznitman in \cite{Sz02}, and Kalikow's condition from \cite{K81}. For simplicity, instead of giving the original formulation of (T') by Sznitman, we will state here the polynomial condition (P) for ballisticity, which was shown in \cite{BDR14} to be equivalent to (T') (and, recently in \cite{GR18}, to
		condition (T), which is the requirement that the exit probability from slabs along the opposite side of the direction of
		movement of the random walk, decays exponentially fast). The interested reader is invited to consult a more detailed exposition about such conditions in \cite{LRSS}. For simplicity, we consider only ballisticity in direction $e_1$.

\medskip
\noindent \textbf{Condition (P)}. We will say that condition (P) is satisfied (in direction $e_1$) if there exists $M \geq M_0$ such that
\begin{equation}
\label{eq:condp}
\sup_{x \in B_M^*} P_x\left( X_{T_{B_M}} \notin \partial_+ B_M\right) \leq \frac{1}{M^{K}},
\end{equation} for some $K>15d+5$, where
$$
M_0:= \exp\{100 + 4d(\log \kappa)^2\}
$$ and $\kappa$ is the uniform ellipticity constant of the RWRE, which in our case can be taken to be $\kappa=\frac{1}{4d}$. We mention once again that in \cite{BDR14} it was shown that (P) is equivalent to (T').

\medskip
\noindent \textbf{Introducing Kalikow's walk}. Given a nonempty connected strict subset $B \subsetneq \Z^d$, for each $x \in B$ we define \textit{Kalikow's walk} on $B$ (starting from $x$) as the random walk starting from $x$ which is killed upon exiting $B$ and has transition probabilities determined by the environment $\omega_B^x \in \mathcal{P}^B$ given by 
\begin{equation}
\label{eq:defkalb}
\omega_B^x(y,e):=\frac{ \E( g_B(x,y,\omega)\omega(y,e))}{\E(g_B(x,y,\omega))}.
\end{equation} It is straightforward to check that by the uniform ellipticity of $\PP$ we have 
\begin{equation} \label{eq:gwd}
0< \E(g_B(x,y,\omega)) <+\infty \hspace{1cm}\text{for all $y \in B$,}
\end{equation}
so that the environment $\omega_B^x$ is well-defined. In accordance with our present notation, we will denote the law of Kalikow's walk on $B$ by $P_{x,\omega_B^x}$ and its Green's function by $g_B(x,\cdot,\omega_B^x)$. Given a direction $\ell \in S^{d-1}$, we will say that Kalikow's condition holds (in direction $\ell$) if 
\begin{equation} \label{defepsK}
\varepsilon_K(\ell):=\inf\{ \vec{d}_{B,0}(y) \cdot \ell : B \subsetneq \Z^d \text{ connected with }0,y \in B\} > 0,
\end{equation} where $\vec{d}_{B,0}$ denotes the drift of Kalikow's walk in $B$ at $0$ defined as
$$
\vec{d}_{B,0}(y) := \sum_{e \in V} \omega_B^0(y,e) e.
$$ If Kalikow's condition holds in some direction $\ell \in S^{d-1}$ then the walk $X$ is ballistic in direction $\ell$, see \cite{LRSS} for further details. Furthermore, in \cite{Sz02} it is shown that Kalikow's condition implies (T'). 

\section{Proof of Theorem \ref{theorem1}}

We will divide the proof of Theorem \ref{theorem1} into three parts, each carried out in a separate subsection. Throughout this section we shall assume that the environmental law $\PP$ satisfies Assumptions \ref{assump}. 

\subsection{Lower bound on $\E(G_U[\vec{d}\cdot e_1](0))$}

The first step is to give a lower bound on $\E(G_U[\vec{d}\cdot e_1](0))$, the expectation of Green's operator on the local drift at $0$. The precise bound we need is contained in the following proposition, which is a generalization of \cite[Proposition 3.1]{Sz03}.

\begin{proposition}
	\label{green-estimates-1} Suppose that $d\ge 3$ and Assumptions \ref{assump} hold. Then, there exist constants $c_1,c_2>0$ depending only on $d$ such that if $\epsilon \leq \frac{1}{8d}$ then for any integer $L \geq 2$ satisfying $\epsilon L < c_1$ one has that 
	\begin{equation} \label{eq:boundexp}
	\E(G_U[\vec{d}\cdot e_1](0)) \geq \frac{2}{5}d \lambda L^2
	\end{equation} provided that 
	\begin{equation} \label{eq:lambda1}
	\lambda \geq c_2 \sigma^2 \left(\epsilon \log L + \frac{1}{L}\right).
	\end{equation}
\end{proposition}

\begin{obs} Since $\sigma \leq \epsilon$ by standard properties of the $L^p$ norms, Proposition \ref{green-estimates-1} is indeed a generalization of \cite[Proposition 3.1]{Sz03} because the required lower bound on the drift $\lambda$ in \eqref{eq:lambda1} is now smaller. 
\end{obs}

To prove Proposition \ref{green-estimates-1}, we will follow the approach used in the proof of \cite[Proposition 3.1]{Sz03}, albeit with some modifications. To begin, for $x \in \Z^d$, $e \in V$ and $\omega \in \Omega_{\epsilon}$ let us set 
\begin{equation}\label{eq:deftildes}
\widetilde{\delta}(x,e) := \omega(x,e) - \E(\omega(x,e)) \hspace{1cm}\text{ and }\hspace{1cm}\widetilde{d}(x,\omega):=\vec{d}(x,\omega) - \E(\vec{d}(x,\cdot)).
\end{equation} Observe that, if $\mathbbm{1}$ denotes the function constantly equal to $1$ on $\Z^d$, since by definition of $\widetilde{d}$ one has
$$
G_U[\vec{d}\cdot e_1](0) = \lambda G_U[\mathbbm{1}](0)+G_U[\widetilde{d}\cdot e_1](0)
$$ and, moreover, $G_U[\mathbbm{1}](0)\geq \frac{4}{5}dL^2$ by \cite[display (3.7)]{Sz03} (the constant $c_1$ appearing in (3.7) of \cite{Sz03} equals $d$ by \cite[display (2.10)]{Sz03}), we see that \eqref{eq:boundexp} will follow if we show that whenever \eqref{eq:lambda1} holds one has 
$$
|\E(G_U[\widetilde{d}\cdot e_1](0))| \leq \frac{2}{5}d\lambda L^2.
$$ Now, a standard Markov chain calculation yields the formula 
\begin{equation}\label{eq:formg}
G_U[\widetilde{d}\cdot e_1](0)= \sum_{x \in U} \frac{P_{0,\omega}(H_x < T_U)}{P_{x,\omega}(\widetilde{H}_x > T_U)}(\widetilde{d}(x)\cdot e_1)
\end{equation} where $T_U$, $H_x$ and $\widetilde{H}_x$ are as in \eqref{eq:deftb} and \eqref{eq:defhs}, respectively. Using the decomposition 
$$
P_{x,\omega}(\widetilde{H}_x > T_U)= P_{x,\bar{\omega}_x}(\widetilde{H}_x > T_U) + \sum_{e \in V} \widetilde{\delta}(x,e) P_{x+e,\omega}(H_x > T_U)
$$ with $\bar{\omega}_x \in \Omega_{\epsilon}$ the environment given by the formula
$$
\bar{\omega}_x(y,e):=\left\{\begin{array}{ll} \omega(y,e) & \text{ if }y \neq x \\ \\ \E(\omega(x,e)) & \text{ if }y=x,\end{array}\right.
$$ we can rewrite \eqref{eq:formg} as 
$$
G_U[\widetilde{d}\cdot e_1](0)= \sum_{x \in U} \frac{P_{0,\omega}(H_x < T_U)}{P_{x,\bar{\omega}_x}(\widetilde{H}_x > T_U)}(\widetilde{d}(x)\cdot e_1) \left( 1 + \sum_{e \in V} \widetilde{\delta}(x,e) \frac{P_{x+e,\omega}(H_x > T_U)}{P_{x,\bar{\omega}_x}(\widetilde{H}_x > T_U)}\right)^{-1}.
$$ But, since for all $e \in V$ we have the inequality
\begin{equation} \label{eq:wbar1}
P_{x+e,\omega}(H_x > T_U) \leq 4d P_{x,\bar{\omega}_x}(\widetilde{H}_x > T_U)
\end{equation} by the uniform ellipticity in \eqref{eq:kappa}, we see that for $\epsilon \leq \frac{1}{8d}$
$$
\left| \sum_{e \in V} \widetilde{\delta}(x,e) \frac{P_{x+e,\omega}(H_x > T_U)}{P_{x,\bar{\omega}_x}(\widetilde{H}_x > T_U)}\right| \leq 4d \epsilon \leq \frac{1}{2}.
$$ In particular, using that $|(1-u)^{-1} -(1+u)| \leq 2u^2$ whenever $|u|\leq\frac{1}{2}$, we may write
\begin{equation}\label{eq:decompG}
G_U[\widetilde{d}\cdot e_1](0)=A+B+C
\end{equation} with 
$$
A:=\sum_{x \in U} \frac{P_{0,\omega}(H_x < T_U)}{P_{x,\bar{\omega}_x}(\widetilde{H}_x > T_U)}(\widetilde{d}(x)\cdot e_1), 
$$
$$
B:= - \sum_{x \in U} \frac{P_{0,\omega}(H_x < T_U)}{P_{x,\bar{\omega}_x}(\widetilde{H}_x > T_U)}(\widetilde{d}(x)\cdot e_1) \sum_{e \in V} \widetilde{\delta}(x,e) \frac{P_{x+e,\omega}(H_x > T_U)}{P_{x,\bar{\omega}_x}(\widetilde{H}_x > T_U)}
$$ and
$$
|C| \leq 2\frac{\epsilon}{d} \sum_{x \in U} \frac{P_{0,\omega}(H_x < T_U)}{P_{x,\bar{\omega}_x}(\widetilde{H}_x > T_U)}\left(\sum_{e \in V} |\widetilde{\delta}(x,e)| \frac{P_{x+e,\omega}(H_x > T_U)}{P_{x,\bar{\omega}_x}(\widetilde{H}_x > T_U)}\right)^2.
$$ 
Note that, since $P_{0,\omega}(H_x < T_U)$ and $P_{x,\bar{\omega}_x}(\widetilde{H}_x> T_U)$ are independent of $\widetilde{d}(x)$, we have $\E(A)=0$. On the other hand, if $c_2$ is chosen large enough then $\E(|C|) \leq \frac{1}{5}d\lambda L^2$. Indeed, from \eqref{eq:wbar1}
we obtain
$$
\left(\sum_{e \in V} |\widetilde{\delta}(x,e)| \frac{P_{x+e,\omega}(H_x > T_U)}{P_{x,\bar{\omega}_x}(\widetilde{H}_x > T_U)}\right)^2 \leq (2d)^2 (4d)^2 \sum_{e \in V} |\widetilde{\delta}(x,e)|^2 = 64d^4 \|\widetilde{\delta}(x)\|_{2}^2
$$ where $\widetilde{\delta}(x):=(\widetilde{\delta}(x,e))_{e \in V}$ and $\|\cdot\|_2$ is the standard $\ell^2$-norm in $\R^d$, so that 
$$
\E(|C|) \leq 128 d^3 \epsilon \sum_{x \in U} \E\left( \frac{P_{0,\omega}(H_x < T_U)}{P_{x,\bar{\omega}_x}(\widetilde{H}_x > T_U)}  \|\widetilde{\delta}(x)\|_2^2\right).
$$ Now, since $P_{0,\omega}(H_x < T_U)$ and $P_{x,\bar{\omega}_x}(\widetilde{H}_x> T_U)$ are also independent of $\widetilde{\delta}(x)$, it follows that
\begin{align*}
\E(|C|) &\leq 128 d^3 \epsilon \sigma^2 \sum_{x \in U} \E\left( \frac{P_{0,\omega}(H_x < T_U)}{P_{x,\bar{\omega}_x}(\widetilde{H}_x > T_U)}\right) \\
& \leq 512 d^4 \epsilon \sigma^2 \sum_{x \in U} \E( g_U(0,x))\\
&=512 d^4 \epsilon \sigma^2 \E(G_U[\mathbbm{1}](0)),
\end{align*} where, to obtain the second inequality, we have used the fact that 
\begin{equation} \label{eq:cota2}
P_{x,\omega}(\widetilde{H}_x > T_U) \leq 4d P_{x,\bar{\omega}_x}(\widetilde{H}_x > T_U)
\end{equation} which follows at once from \eqref{eq:wbar1}. Finally, recalling that $\E(G_U[\mathbbm{1}](0)) \leq \frac{4}{3}dL^2$ whenever $\epsilon L < \frac{3}{4}$ by \cite[Proposition 2.2]{Sz03}, we conclude that
$$
\E(|C|) \leq \frac{2048}{3}\epsilon d^5 \sigma^2 L^2.
$$ Hence, by taking $c_2> 0$ sufficiently large so that $\frac{1}{5}d\lambda \geq \frac{2048}{3} \epsilon d^5 \sigma^2$ when \eqref{eq:lambda1} holds, we see that $\E(|C|) \leq \frac{1}{5}d\lambda L^2$ is satisfied as claimed. Thus, in order to conclude the proof it will suffice to show that
\begin{equation} \label{eq:cota4}
|\E(B)| \leq \frac{1}{5}d\lambda L^2.
\end{equation} To this end, we first notice that, since for all $x \in U$ we have
$$
\sum_{e \in V} \widetilde{\delta}(x,e)=0,
$$ $B$ remains unchanged if we replace the term $P_{x+e,\omega}(H_x > T_U)$ in the innermost sum with 
$$
P_{x+e,\omega}(H_x < T_U)-P_{x+e_1,\omega}(H_x < T_U).
$$ Since the latter term is independent of $\omega(x)$, we obtain that
\begin{align}
\E(|B|)&\leq \sum_{x \in U , e \in V} \E\left(\frac{P_{0,\omega}(H_x <T_U)}{P^2_{x,\bar{\omega}_x}(\widetilde{H}_x > T_U)}|P_{x+e,\omega}(H_x < T_U)-P_{x+e_1,\omega}(H_x < T_U)| |\widetilde{d}(x)\cdot e_1||\widetilde{\delta}(x,e)|\right) \nonumber \\
\nonumber \\
& = \sum_{x \in U , e \in V} \E\left(\frac{P_{0,\omega}(H_x <T_U)}{P^2_{x,\bar{\omega}_x}(\widetilde{H}_x > T_U)}|P_{x+e,\omega}(H_x < T_U)-P_{x+e_1,\omega}(H_x < T_U)|\right) \E(|\widetilde{d}(x)\cdot e_1||\widetilde{\delta}(x,e)|). \label{eq:cota3}
\end{align}
Now, on the one hand we have $|\widetilde{d}(x)\cdot e_1|=|\widetilde{\delta}(x,e_1)-\widetilde{\delta}(x,-e_1)|\leq |\widetilde{\delta}(x,e_1)|+|\widetilde{\delta}(x,-e_1)|$ so that by the Cauchy-Schwarz inequality
$$
\E(|\widetilde{d}(x)\cdot e_1||\widetilde{\delta}(x,e)|) \leq 2\sigma^2
$$ for all $x \in U, e \in V$. On the other hand, by \eqref{eq:cota2} and the fact that for all $y,x \in U$
$$
g_U(y,x,\omega)=\frac{P_{y,\omega}(H_x < T_U)}{P_{x,\omega}(\widetilde{H}_x > T_U)}
$$ we have that the leftmost expectation in \eqref{eq:cota3} can be bounded from above by 
$$
(4d)^2 g_U(0,x,\omega)|g_U(x+e,x,\omega)-g_U(x+e_1,x,\omega)|
$$ so that in order to show \eqref{eq:cota4} it will suffice to find some constant $c=c(d)>0$ so that
\begin{equation} \label{eq:cota5}
\sup_{e \in V, \omega \in \Omega_{\epsilon}} \sum_{x \in U} g_U(0,x,\omega)|g_U(x+e,x,\omega)-g_U(x+e_1,x,\omega)| \leq c\left( \epsilon \log L + \frac{1}{L}\right)L^2
\end{equation} and then choose the constant $c_2$ from the statement of the proposition large enough. But \eqref{eq:cota5} can be shown exactly as in the proof of \cite[Proposition 3.1, display (3.22)]{Sz03}. We omit the details.

\subsection{Controlling the fluctuations of $G_U[\vec{d}\cdot e_1](0)$}

The next step is to obtain an analogue of the result in \cite[Proposition 3.2]{Sz03}, which gives control on the fluctuations of $G_U[\vec{d}\cdot e_1](0)$, for the $L^2(\PP)$-setting. The key to this extension is the following lemma, a variation of Azuma-Hoeffding's inequality incorporating the conditional variance of the underlying martingale difference.  

\begin{lemma}\label{lemma:azuma} Let $(F_n)_{n \in \N_0}$ be a martingale with respect to some filtration $(\mathcal{G}_n)_{n \in \N_0}$ satisfying $\mathcal{G}_0=\{\emptyset,\Omega\}$. If there exist a constant $b > 0$ and a sequence $(v_n)_{n \in \N} \subseteq \R_{\geq 0}$ such that for every $n \in \N$
	$$
	|F_n-F_{n-1}| \leq b \hspace{2cm}\text{ and }\hspace{2cm} \E((F_n-F_{n-1})^2|\mathcal{G}_{n-1}) \leq v_n^2,
	$$ then for any $u > 0$ one has that
	$$
	\max\left\{P\left( \liminf_{n \rightarrow +\infty} F_n > F_0 + u\right),P\left( \liminf_{n \rightarrow +\infty} F_n < F_0 - u\right) \right\} \leq \exp\left\{ -\frac{1}{2}\cdot \frac{u^2}{\sum_{n \in \N}v_n^2+\frac{1}{3}ub}\right\}.
	$$
\end{lemma}

\begin{proof} Call $F_\infty:=\liminf_{n \rightarrow +\infty} F_n$. Upon noticing that
	$$
	P( F_\infty - F_0 > u) \leq \limsup_{n \rightarrow +\infty} P\left( F_k - F_0 > u \text{ and }\sum_{i=1}^k \E((F_i-F_{i-1})^2|\mathcal{G}_{i-1}) \leq \sum_{i \in \N}v_i^2 \text{ for some }k \in [1,n]\right),
	$$ from Theorem 2.1 and Remark 2.1 on \cite{FGL12} one obtains the bound
	$$
	P( F_\infty - F_0 > u) \leq \exp\left\{ -\frac{1}{2}\cdot \frac{u^2}{\sum_{n \in \N}v_n^2+\frac{1}{3}ub}\right\}.
	$$ The remaining bound follows by symmetry, working with $-F_n$ instead of $F_n$.
\end{proof}

The variant of \cite[Proposition 3.2]{Sz03} we require for our purposes is the following.

\begin{proposition} 
	\label{control} Suppose that $d \geq 3$. Then, for any fixed $\alpha \in [0,1)$ there exist constants $c_3,c_4>0$ depending only on $d$ and $\alpha$ such that, for any integer $L\geq 2$ satisfying $\epsilon L < c_3$, one has that 
	\begin{equation}
	\label{green-variance}
	\PP\left(\left|
	G_U[\vec{d}\cdot e_1](0)- \E(G_U[\vec{d}\cdot e_1](0))\right| \geq u\right) 
	\le  2\exp \left\{-\frac{1}{c_{4}}\cdot \frac{u^2}{c_{\alpha,L} \sigma^2 + u\epsilon}\right\}
	\end{equation} for all $u \geq 0$, where 
	\begin{equation}
	\nonumber
	c_{\alpha,L}:=
	\begin{cases}
	L^{1+(2(1-\alpha)/(2-\alpha))} & \quad{\rm if}\ d=3\\
	L^{(4(1-\alpha)/(2-\alpha))} & \quad{\rm if}\ d=4\\ 1
	& \quad{\rm if}\ d\ge 5\ {\rm and}\ \alpha\ge\frac{4}{5}.
	\end{cases}
	\end{equation}
\end{proposition}


\begin{proof}
	We follow the proof of \cite[Proposition 3.2]{Sz03}, applying the martingale method introduced therein.
	First, let us enumerate the elements of $U$ as $U:=\{x_n:n \in \N\}$.
	Now, define the filtration
	$$
	\mathcal G_n:=
	\begin{cases}
	\sigma(\omega(x_1),\ldots,\omega(x_n))& \qquad \text{ if }n\ge 1\\
	\{\emptyset,\Omega\}&\qquad \text{ if }n=0
	\end{cases}
	$$
	and also the bounded (see \cite[Proposition 2.2]{Sz03} for a justification) $\mathcal G_n$-martingale $(F_n)_{n \in \N_0}$ given for each $n \in \N_0$ by 
	$$
	F_n:=\mathbb E(G_U[\vec{d} \cdot e_1](0)|\mathcal G_n).
	$$ Since $F_\infty:=\lim_{n \rightarrow +\infty} F_n = G_U[\vec{d}\cdot e_1](0)$ and $F_0=\E(G_U[\vec{d}\cdot e_1](0))$, by Lemma \ref{lemma:azuma} we get
	$$
	\PP\left[\left|
	G_U[\vec{d}(\cdot,\omega)\cdot e_1](0)-\E(G_U[\vec{d}\cdot e_1](0))\right|\ge u\right]
	\le  2\exp\left\{-\frac{1}{2}\cdot\frac{u^2}{\sum_{n \in \N} v_n^2 + ub}\right\}
	$$ for all $b$ and $(v_n)_{n \in \N}$ with $|F_n-F_{n-1}|\leq b$ and $\E((F_n-F_{n-1})^2|\mathcal{G}_{n-1}) \leq v_n^2$ for every $n$. 
	
	Thus, let us find such $b$ and $v_n$. To this end, for each
	$n \in \N$ and all environments $\omega,\omega'\in\Omega_\epsilon$ with $\omega \equiv \omega'$ off $x_n$, i.e. which coincide at every $x_i$ with $i \neq n$, let us define
	$$
	\Gamma_n(\omega,\omega'):=G_U[\vec{d}\cdot e_1](0,\omega)-G_U[\vec{d}\cdot e_1](0,\omega').
	$$ If $\mu$ denotes the single site distribution of $\omega_U=(\omega(x_i))_{i \in \N}$ under $\PP$, then it is simple to see that
	\begin{equation}
	\label{eq:rep}
	F_n(\omega) - F_{n-1}(\omega) = \int \left[ \int \Gamma_n(\omega,\omega') d\mu(\omega'(x_n))\right]d\mu^{\otimes (i>n)}(\omega_{(i>n)}),
	\end{equation} where $\omega_{(i>n)} :=(\omega(x_i))_{i > n}$, $\mu^{\otimes (i>n)}$ denotes its distribution under the law $\PP$ and 
	$\omega' \equiv \omega$ off $x_n$. Furthermore, it follows from the proof of \cite[Proposition 3.2]{Sz03} that for $L \geq 2$ as in the statement of the proposition one has
	\begin{equation}
	\label{eq:rep2}
	|\Gamma_n(\omega,\omega')| \leq c(d) g_{0,U}(0,x_n)^{\frac{1}{2-\alpha}} \sum_{e \in V}|\omega(x_n,e)-\omega'(x_n,e)|
	\end{equation}
	for some constant $c(d) > 0$. In particular, since $g_{0,U}(0,\cdot)$ is a bounded function, \eqref{eq:rep} and \eqref{eq:rep2} together yield that 
	$$
	|F_n-F_{n-1}| \leq \sup_{\omega \equiv \omega' \text{ off }x_n} |\Gamma_n(\omega,\omega')| \leq c(d,\alpha) \epsilon=:b
	$$ for some constant $c(d,\alpha)>0$. Now, on the other hand, by using Fubini's theorem and Minkowski's integral inequality (the latter applied twice), we obtain 
	\begin{align*}
	\E((F_n-F_{n-1})^2|\mathcal{G}_{n-1}) & = \int \left(\int \left[ \int \Gamma_n(\omega,\omega') d\mu(\omega'(x_n))\right]d\mu^{\otimes (i>n)}(\omega_{(i>n)})\right)^2 d\mu(\omega(x_n))\\
	& = \int \left(\int \left[ \int \Gamma_n(\omega,\omega') d\mu^{\otimes (i>n)}(\omega_{(i>n)})\right]d\mu(\omega'(x_n))\right)^2 d\mu(\omega(x_n))\\
	& \leq \left(\int \left[\int \left( \int \Gamma_n(\omega,\omega') d\mu^{\otimes (i>n)}(\omega_{(i>n)})\right)^2 d\mu(\omega(x_n))\right]^{\frac{1}{2}}d\mu(\omega'(x_n))\right)^2\\
	& \leq \left(\int \int \left[ \int \Gamma^2_n(\omega,\omega') d\mu(\omega(x_n)) \right]^{\frac{1}{2}} d\mu^{\otimes (i>n)}(\omega_{(i>n)})d\mu(\omega'(x_n))\right)^2.
	\end{align*} 
	But by \eqref{eq:rep2}
	we have that
	$$
	\Gamma^2_n(\omega,\omega')\leq c(d) g_{0,U}(0,x_n)^{\frac{2}{2-\alpha}}\|\omega(x_n)-\omega'(x_n)\|_2^2,$$ so that
	\begin{align*}
	\E((F_n-F_{n-1})^2|\mathcal{G}_{n-1}) &\leq c(d) g_{0,U}(0,x_n)^{\frac{2}{2-\alpha}}\left( \int \left[\int \|\omega(x_n)-\omega'(x_n)\|_2^2 d\mu(\omega(x_n))\right]^{\frac{1}{2}}d\mu(\omega'(x_n))\right)^2\\
	& \leq c'(d) g_{0,U}(0,x_n)^{\frac{2}{2-\alpha}}\left( \int \sqrt{ \E(\|\widetilde{\delta}(x_n)\|_2^2) +\|\widetilde{\delta}'(x_n)\|_2^2}d\mu(\omega'(x_n))\right)^2\\
	& \leq c'(d) g_{0,U}(0,x_n)^{\frac{2}{2-\alpha}}\left( \sqrt{ \E(\|\widetilde{\delta}(x_n)\|_2^2)} + \E(\|\widetilde{\delta}'(x_n)\|_2)\right)^2\\
	& \leq c''(d)  g_{0,U}(0,x_n)^{\frac{2}{2-\alpha}}\sigma^2=:v_n^2,
	\end{align*} where $\widetilde{\delta}(x_n):=(\omega(x_n,e)-\E(\omega(x_n,e)))_{e \in V}$ and $\widetilde{\delta}'(x_n)$ is defined analogously but with $\omega'$ instead of $\omega$.
	Finally, it follows from \cite[displays (3.34)-(3.35)]{Sz03} that
	$$
	\sum_{n \in \N} g_{0,U}(0,x_n)^{\frac{2}{2-\alpha}} \leq \widetilde{c}(d,\alpha) c_{\alpha,L}
	$$ for some $\widetilde{c}(d,\alpha)>0$ and with $c_{\alpha,L}$ the constant from the statement of the proposition, from where the result now readily follows.
\end{proof}

\subsection{Conclusion of the proof} We now finish the proof of Theorem \ref{theorem1} by checking that (T') holds. We will split the proof in two cases, depending on whether $\sigma < \epsilon^{1.5}$ or $\sigma \geq \epsilon^{1.5}$. 

Notice that if $\sigma < \epsilon^{1.5}$ then $\lambda \geq \sigma \epsilon^{1.5-\eta}$ implies that $\lambda \geq \sigma^2 \epsilon^{-\eta}$. In particular, given $\eta \in (0,1)$ there exists $\epsilon_0(\eta,d) \in (0,1)$ such that $\lambda > 4d\sigma^2(1+9\epsilon)$ if $\epsilon \leq \epsilon_0$. Thus, in this case that (T') holds follows from Theorem \ref{theorem2}, which we will prove in Section 4 in an independent manner.

Hence, it suffices to prove the result assuming that $\sigma \geq \epsilon^{1.5-\eta}$. Note that, in this case, we have the inequality
$$
\lambda \geq \max\{ \sigma \epsilon^{1.5-\eta},\epsilon^{3-\eta}\}=:\lambda_0.
$$ We will prove the result by verifying the polynomial condition (P) in \eqref{eq:condp}. To this end, we will use Propositions \ref{green-estimates-1} and \ref{control} above, as well as the estimates derived in \cite{Sz03} to establish the validity of the so-called effective criterion. 

First, let us fix a constant $\theta > 0$ with value to be specified later and set
$$
L:=2[\theta \epsilon^{-1}]\hspace{2cm}\text{ and }\hspace{2cm}M:=L^4.
$$ Now, consider the box $B$ given by 
\begin{equation}\label{eq:defb}
B:=(-M,M) \times \left(-\frac{1}{4}M^3,\frac{1}{4}M^3\right)^{d-1}
\end{equation} and define its frontal part $\partial_+ B$  by analogy with Section \ref{sec:GN}. Note that if for each $x \in B^*_{M}$ (with $B_M$ as in \eqref{beeme}) we consider $B(x):=B+x$, i.e the translate of $B$ centered at $x$, then by construction of $B$ we have that for any $\omega \in \Omega$
\begin{equation}\label{eq:pcb}
P_{x,\omega}\left( X_{T_{B_{M}}} \notin \partial_+ B_{M}\right) \leq P_{x,\omega}\left( X_{T_{B(x)}} \notin \partial_+ B(x)\right).
\end{equation} Thus, it follows from the translation invariance of $\PP$ that, in order to check that condition (P) holds, it will suffice to show that 
\begin{equation}\label{eq:pc1}
P_0\left( X_{T_B} \notin \partial_+ B\right) \leq e^{-c\epsilon^{-1}}
\end{equation}for some constant $c=c(d,\eta)> 0$ if $\epsilon$ is taken sufficiently small (but depending only on $d$ and $\eta$). To do this, the idea is to exploit the estimates derived in \cite[Section 4]{Sz03}. Indeed, if for any $\omega \in \Omega$ we define
$$
q_B(\omega):=P_{0,\omega}\left( X_{T_B} \notin \partial_+ B\right) \hspace{1cm}\text{ and }\hspace{1cm}\rho_B(\omega):=\frac{q_B(\omega)}{1-q_B(\omega)}
$$ then we have 
$$
P_0(X_{T_B} \notin \partial_+ B) = \E(q_B) \leq \E(\sqrt{q_B}) \leq \E(\sqrt{\rho_B}).
$$ Now, by following the arguments of \cite[Section 4]{Sz03}, we see that there exists  $\widetilde{\epsilon}_1=\widetilde{\epsilon}_1(d,\eta)> 0$ such that if $\epsilon L < \frac{3}{4}$ and $\epsilon \leq \widetilde{\epsilon}_1$ then  
\begin{equation}
\label{eq:preboundrhob}
\E(\sqrt{\rho_B}) \leq \frac{2(\E(\widehat{\rho}(0)))^{\frac{M}{2L}}}{(1-(\E(\widehat{\rho}(0)))^{\frac{1}{2}})_+} +2d\exp\left( M \left[\frac{\log 4d}{2}- 50\frac{\log 4d}{\log 2}\left(p-\frac{7}{100}\right)^2_+\right]\right)
\end{equation} where 
$$
\widehat{\rho}(0):=\sup\left\{\frac{1-\frac{1}{L}G_U[\vec{d}\cdot e_1](x)}{1+\frac{1}{L}G_U[\vec{d}\cdot e_1](x)}: x \cdot e_1 = 0\,,\,\sup_{2\leq j \leq d} |x\cdot e_j|<\frac{1}{4}M^3\right\}
$$ and, furthermore, $\widehat{\rho}(0)$ and $p$ are such that
\begin{equation}
\label{eq:preboundrhop}
\widehat{\rho}(0)\leq 3\hspace{2cm}\text{ and }\hspace{2cm}p \geq 1 - M^{2d}\PP\left( G_U[\vec{d}\cdot e_1](0) \leq \lambda_0 L + \frac{4}{L^2}\right).
\end{equation} 
Thus, if we manage to show that for $\epsilon$ sufficiently small we have the more precise estimates
\begin{equation}\label{eq:boundrhop}
\E(\widehat{\rho}(0)) \leq 1 - \frac{1}{10}d\lambda_0 L\hspace{2cm}\text{ and }\hspace{2cm}
p \geq \frac{3}{4},
\end{equation} a straightforward computation would then yield that, under \eqref{eq:preboundrhob} and \eqref{eq:boundrhop}, one has
$$
\E(\sqrt{\rho_B}) \leq \frac{80}{d\lambda_0 L}\exp\left(-\frac{d}{20}\lambda_0 M\right) + 2d\exp\left(-(\log 4d) M \right), 
$$ from where \eqref{eq:pc1} now immediately follows.

Hence, let us show \eqref{eq:boundrhop}. First, choose $\alpha \in (0,1)$ such that 
$$
2\cdot\frac{1-\alpha}{2-\alpha} = \eta,
$$ and let us fix the value of $\theta$ so that $\theta < \frac{1}{2}\min\{\frac{3}{4},c_1,c_3\}$, where $c_1,c_3>0$ are  the constants from Propositions \ref{green-estimates-1} and \ref{control}, respectively. Observe that, given $\alpha$ and $\theta$ (which depend only on $d$ and $\eta$), there exists $\widetilde{\epsilon}_2 = \widetilde{\epsilon}_2(d,\eta) \in (0,\widetilde{\epsilon}_1)$ such that if $\epsilon \leq \widetilde{\epsilon}_2$ then $L \geq 2$, $\epsilon L < \min\{\frac{3}{4},c_1,c_3\}$ and \eqref{eq:lambda1} holds, so that \eqref{eq:boundexp},\eqref{green-variance},\eqref{eq:preboundrhob} and \eqref{eq:preboundrhop} all hold as well.

Now, using \eqref{eq:boundexp},\eqref{green-variance},\eqref{eq:preboundrhop} and the fact that $\lambda \geq \lambda_0$, a straightforward computation yields 
\begin{align*}
\E(\widehat{\rho}(0)) &\leq \frac{1-\frac{1}{5}d\lambda_0 L}{1+\frac{1}{5}d\lambda_0 L} + 3\PP\left( \inf_{\substack{ x \cdot e_1 = 0 \\ \sup_{2 \leq j \leq d}|x\cdot e_j|<\frac{1}{4}M^3}}\left(G_U[\vec{d}\cdot e_1](x) - \E(G_U[\vec{d}\cdot e_1](x))\right) \leq -\frac{1}{5}d\lambda_0 L^2 \right)\\
& \leq 1-\frac{1}{5}d\lambda_0L + 6\left(\frac{M^3}{2}\right)^{d-1} \PP\left( G_U[\vec{d}\cdot e_1](0) - \E(G_U[\vec{d}\cdot e_1](0)) \leq -\frac{1}{5}d\lambda_0 L^2 \right)\\
& \leq 1-\frac{1}{5}d\lambda_0L + 12\left(\frac{M^3}{2}\right)^{d-1} \exp\left\{-\frac{1}{c_5} \cdot \frac{\lambda_0^2 L^4}{c_{\alpha,L}\sigma^2+\lambda_0L} \right\}\\
& \leq 1 - \frac{1}{10}d\lambda_0L
\end{align*} if $\epsilon$ sufficiently small (depending only on $d$ and $\eta$), where $c_5=c_5(d,\eta)>0$ is just some constant. Similarly, if $\epsilon$ is sufficiently small then $\lambda_0 L + \frac{4}{L^2} \leq \frac{1}{10}d\lambda_0 L^2$, so that by \eqref{eq:boundexp}, \eqref{green-variance} and \eqref{eq:preboundrhop}, we obtain
$$
p \geq 1 - M^{2d}\PP\left( G_U[\vec{d}\cdot e_1](0) - \E(G_U[\vec{d}\cdot e_1](0)) \leq -\frac{1}{10}d\lambda_0L\right) \geq \frac{3}{4}
$$ provided that $\epsilon$ is sufficiently small (depending only on $d$ and $\eta$). Now, by the discussion above, we conclude that there exists $\epsilon_0=\epsilon_0(d,\eta) \in (0,\min\{\widetilde{\epsilon}_1,\widetilde{\epsilon}_2\})$ such that if $\epsilon \leq \epsilon_0$ then (P) holds. This finishes the proof.

\begin{obs} Notice that the assumption $d=3$ was not required anywhere throughout the proof, only that $d \geq 3$ for Propositions \ref{green-estimates-1} and \ref{control}. Indeed, the very same argument also works for $d \geq 4$, giving that
	$$
	\lambda \geq \sigma \epsilon^{1.5-\eta}
	$$ implies the validity of condition (T'). However, for $d \geq 4$ this does not give any additional ballisticity results with respect to those already given by \cite{Sz03} and Theorem \ref{theorem2}. This is the reason why we chose to state Theorem \ref{theorem1} only for $d=3$ in the Introduction. 
\end{obs}

\section{Proof of Theorem \ref{theorem2}}

We begin the proof by recalling the following explicit formula for Kalikow's drift $\vec{d}_{B,x}$ proved in \cite[Lemma 9]{LRSS}: for all $y \in B$ we have 
\begin{equation} \label{eq:driftkal}
\vec{d}_{B,x}(y) \cdot e_1 = \frac{\E\left( \frac{\vec{d}(y,\omega) \cdot e_1}{\sum_{e \in V}\omega(y,e)f_{B,x}(y,y+e,\omega)}\right)}{\E\left( \frac{1}{\sum_{e \in V} \omega(y,e)f_{B,x}(y,y+e,\omega)}\right)},
\end{equation} where $f_{B,x}$ is given by
$$
f_{B,x}(y,z,\omega):=\frac{P_{z,\omega}(T_B \leq H_y)}{P_{x,\omega}(H_y < T_B)}
$$ with $T_B$ and $H_y$ defined as in \eqref{eq:deftb} and \eqref{eq:defhs}, respectively.

Now, for each $e \in V$ let us decompose $\omega(y,e)=\bar{\omega}(y,e)+\widetilde{\delta}(y,e)$ with
$$
\bar{\omega}(y,e):=\E(\omega(y,e)) \hspace{2cm}\text{ and }\hspace{2cm}\widetilde{\delta}(y,e)=\omega(y,e)-\E(\omega(y,e)),
$$ and write
$$
s(y,\omega):= \sum_{e \in V} \bar{\omega}(y,e)f_{B,x}(y,y+e,\omega) \hspace{1cm}\text{ and }\hspace{1cm} r(y,\omega):= \sum_{e \in V} \widetilde{\delta}(y,e)f_{B,x}(y,y+e,\omega). 
$$ Then, using the second order expansion $\frac{1}{a+b}=\frac{1}{a}-\frac{b}{a^2}+\frac{b^2}{(a+\xi)^3}$ for some $\xi \in (\min\{b,0\},\max\{0,b\})$, one can write the numerator in \eqref{eq:driftkal} as
$$
\E\left( \frac{\vec{d}(y,\omega) \cdot e_1}{s(y,\omega)}\right) - \E\left( \frac{(\vec{d}(y,\omega) \cdot e_1) r(y,\omega)}{s^2(y,\omega)}\right) + \E\left( \frac{(\vec{d}(y,\omega) \cdot e_1) r^2(y,\omega)}{(s(y,\omega)+\xi)^3}\right)
$$ for some $\xi \in (\min\{r(y,\omega),0\},\max\{0,r(y,\omega)\})$. Being $(f_{B,x}(y,y+e,\omega))_{e \in V}$ independent of $\omega(y)$, it follows that 
$$
\E\left( \frac{\vec{d}(y,\omega) \cdot e_1}{s(y,\omega)}\right) = \lambda \E\left( \frac{1}{s(y,\omega)}\right)
$$ and also that 
\begin{equation}
\label{eq:0bound}
\E\left( \frac{(\vec{d}(y,\omega) \cdot e_1) r(y,\omega)}{s^2(y,\omega)}\right) = \E\left( \frac{\sum_{e \in V} \E((\vec{d}(y,\omega) \cdot e_1) \widetilde{\delta}(y,e)) f_{B,x}(y,y+e,\omega)}{s^2(y,\omega)}\right).
\end{equation} Furthermore, since $\E(\widetilde{\delta}(y,e))=0$ for all $e \in V$, we have 
$$
\E((\vec{d}(y,\omega) \cdot e_1) \widetilde{\delta}(y,e))= \E((\widetilde{\delta}(y,e_1)-\widetilde{\delta}(y,-e_1))\widetilde{\delta}(y,e))
$$ so that, by Cauchy-Schwarz inequality, we obtain the bound
\begin{equation}
\label{eq:boundsigma}
|\E((\vec{d}(y,\omega) \cdot e_1) \widetilde{\delta}(y,e))| \leq 2\sigma^2.
\end{equation}
Now, since $\bar{\omega}(y,e) \geq \frac{1}{2d}(1-\frac{\epsilon}{2})$ for all $e \in V$ by Assumption (A2), a straightforward computation using \eqref{eq:boundsigma} then yields the estimate
\begin{equation}
\label{eq:boundsigma2}
\left|\E\left( \frac{(\vec{d}(y,\omega) \cdot e_1) r(y,\omega)}{s^2(y,\omega)}\right)\right| \leq \frac{4d}{1-\frac{\epsilon}{2}}\sigma^2\E\left( \frac{1}{s(y,\omega)}\right).
\end{equation} Finally, using \eqref{epsilon-condition} again, we have 
\begin{equation}\label{eq:1bound}
\left|\E\left( \frac{(\vec{d}(y,\omega) \cdot e_1) r^2(y,\omega)}{(s(y,\omega)+\xi)^3}\right)\right|  \leq \frac{\epsilon}{2d} \E\left( \frac{ \left(\sum_{e \in V} |\widetilde{\delta}(y,e)|f_{B,x}(y,y+e,\omega)\right)^2}{|s(y,\omega)+\xi|^3}\right),
\end{equation} where the expectation on the right-hand side of \eqref{eq:1bound} can be further bounded from above by
\begin{equation}\label{eq:2bound}
\E\left( \frac{ \left(\sum_{e \in V} |\widetilde{\delta}(y,e)|f_{B,x}(y,y+e,\omega)\right)^2}{s^3(y,\omega)}\right) + \E\left( \frac{ \left(\sum_{e \in V} |\widetilde{\delta}(y,e)|f_{B,x}(y,y+e,\omega)\right)^2}{(\sum_{e \in V} \omega(y,e)f_B(y,y+e,\omega))^3}\right).
\end{equation} By an argument similar to the one yielding \eqref{eq:boundsigma2}, we obtain 
$$
\E\left( \frac{ \left(\sum_{e \in V} |\widetilde{\delta}(y,e)|f_{B,x}(y,y+e,\omega)\right)^2}{s^3(y,\omega)}\right) \leq \frac{4d^2}{(1-\frac{\epsilon}{2})^2}\sigma^2 \E\left( \frac{1}{s(y,\omega)}\right).
$$ On the other hand, since by \eqref{epsilon-condition} we have that $\omega(y,e) \geq c_\epsilon\bar{\omega}(e)$ for all $e \in V$ and $c_\epsilon:=\frac{1-\frac{\epsilon}{2}}{1+\frac{\epsilon}{2}} \in (0,1)$, the same argument now yields 
$$
\E\left( \frac{ \left(\sum_{e \in V} |\widetilde{\delta}(y,e)|f_{B,x}(y,y+e,\omega)\right)^2}{(\sum_{e \in V} \omega(y,e)f_B(y,y+e,\omega))^3}\right) \leq 4d^2\frac{(1+\frac{\epsilon}{2})}{(1-\frac{\epsilon}{2})^3} \sigma^2 \E\left( \frac{1}{s(y,\omega)}\right)
$$ so that
$$
\left|\E\left( \frac{(\vec{d}(y,\omega) \cdot e_1) r^2(y,\omega)}{(s(y,\omega)+\xi)^3}\right)\right| \leq \frac{2d}{(1-\frac{\epsilon}{2})^2}\left[1+\frac{1+\frac{\epsilon}{2}}{1-\frac{\epsilon}{2}}\right] \epsilon \sigma^2 \E\left( \frac{1}{s(y,\omega)}\right).
$$ Finally, by repeating the same procedure but with  $\vec{d}(y,\omega)$ replaced by $1$, the denominator in \eqref{eq:driftkal} can be written as
$$
\E\left( \frac{1}{s(y,\omega)}\right) + \E\left( \frac{r^2(y,\omega)}{(s(y,\omega)+\xi)^3}\right)
$$ where
$$
\left|\E\left( \frac{r^2(y,\omega)}{(s(y,\omega)+\xi)^3}\right)\right| \leq \frac{4d^2}{(1-\frac{\epsilon}{2})^2}\left[1+\frac{1+\frac{\epsilon}{2}}{1-\frac{\epsilon}{2}}\right] \sigma^2 \E\left( \frac{1}{s(y,\omega)}\right).
$$ Gathering all estimates obtained, we conclude that
$$ 
\vec{d}_{B,x}(y) \cdot e_1 \geq \frac{ \lambda - 4d \sigma^2 \left(\frac{1}{1-\frac{\epsilon}{2}}+\frac{\epsilon}{(1-\frac{\epsilon}{2})^3}\right)}{1+\frac{8d^2}{(1-\frac{\epsilon}{2})^3}\sigma^2}
$$ so that $\varepsilon_K> 0$ (where $\epsilon_K$ is the one defined in \eqref{defepsK}) provided that
\begin{equation}\label{eq:drifcota}
\lambda > 4d \sigma^2 \left(\frac{1}{1-\frac{\epsilon}{2}}+\frac{\epsilon}{(1-\frac{\epsilon}{2})^3}\right).
\end{equation} Since a simple computation shows that for all $\epsilon \in (0,1)$ we have
$$
\frac{1}{1-\frac{\epsilon}{2}}+\frac{\epsilon}{(1-\frac{\epsilon}{2})^3} \leq 1 + 9\epsilon,
$$ \eqref{eq:drifcota} then immediately gives the result.

\section{Proof of Theorem \ref{theo:nokal}}

We follow the approach used in the proof of \cite[Theorem 5.1]{Sz03}, introducing some modifications along the way. Define the sets
$$
U_+:=\{ y \in \Z^d : y \cdot e_1 \geq 0 \} \hspace{1cm}\text{ and }\hspace{1cm}U_-:=\{ y \in \Z^d : y \cdot e_1 \leq 0\}.
$$ We will check that 
\begin{equation}
\label{eq:drifco2}
\vec{d}_{U_-,0}(0)\cdot e_1 < 0 < \vec{d}_{U_+,0}(0)\cdot e_1
\end{equation} and also that 
\begin{equation} \label{eq:drifco}
\vec{d}_{U_+,0}(0)\cdot e =\vec{d}_{U_-,0}(0) \cdot e = 0
\end{equation} for all $e \in V$ with $e \cdot e_1 = 0$, which together imply that Kalikow's condition fails for all directions. 

To show \eqref{eq:drifco2}, let us first notice that, by definition of $\vec{d}_{U_\pm,0}(0)$ and \eqref{eq:gwd}, it will be enough to check that
\begin{equation}\label{eq:desig}
\E(g_{U_-}(0,0,\cdot)(\vec{d}\cdot e_1)(0,\cdot))<0<\E(g_{U_+}(0,0,\cdot)(\vec{d}\cdot e_1)(0,\cdot)).
\end{equation}
To this end, let us write 
$$
\E(g_{U_\pm}(0,0,\cdot)(\vec{d}\cdot e_1)(0,\cdot)) =\lambda\E(g_{U_\pm}(0,0,\cdot)) + \E(g_{U_\pm}(0,0,\cdot)(\widetilde{d}\cdot e_1)(0,\cdot)),
$$ for $\widetilde{d}$ defined as in \eqref{eq:deftildes}. Observe that
$$
0 \leq \E(g_{U_\pm}(0,0,\cdot)) = \E\left(\frac{1}{P_{0,\cdot}(\widetilde{H}_0> T_{U_\pm} )}\right) \leq 4d
$$ by uniform ellipticity and also that, by manipulations analogous to those performed to obtain \eqref{eq:decompG}, for $\epsilon \leq \frac{1}{8d}$ we may write
$$
\E(g_{U_\pm}(0,0,\cdot)(\widetilde{d}\cdot e_1)(0,\cdot))=\beta_{U_\pm}+\gamma_{U_\pm},
$$ where 
$$
\beta_{U_\pm}:=- \E\left(\frac{\widetilde{d}(0)\cdot e_1}{P_{0,\bar{\omega}_0}(\widetilde{H}_0 > T_{U_\pm})} \sum_{e \in V} \widetilde{\delta}(0,e) \frac{P_{e,\omega}(H_0 > T_{U_\pm})}{P_{0,\bar{\omega}_0}(\widetilde{H}_0 > T_{U_\pm})}\right)
$$ and
\begin{equation}\label{eq:gamma}
|\gamma_{U_\pm}| \leq 2 \frac{\epsilon}{d} \E\left(\frac{1}{P_{0,\bar{\omega}_0}(\widetilde{H}_0 > T_{U_\pm})} \left( \sum_{e \in V} |\widetilde{\delta}(0,e)|\frac{P_{e,\omega}(H_0 > T_{U_\pm})}{P_{0,\bar{\omega}_0}(\widetilde{H}_0 > T_{U_\pm})}\right)^2\right).
\end{equation} Since $P_{e,\omega}(H_0 > T_{U_\pm})$ and $P_{0,\bar{\omega}_0}(\widetilde{H}_0 > T_{U_\pm})$ are independent of $\omega(0)$, it follows that 
\begin{equation}\label{eq:beta}
\beta_{U_\pm}=-\sum_{e \in V} \E\left( \frac{P_{e,\omega}(H_0 > T_{U_\pm})}{P^2_{0,\bar{\omega}_0}(\widetilde{H}_0 > T_{U_\pm})}\right) \E((\widetilde{d}(0)\cdot e_1)\widetilde{\delta}(0,e) ).
\end{equation} Furthermore, due to our assumption (K2), the leftmost expectation in \eqref{eq:beta} is constant over all $e \in V$ with $e \cdot e_1 = 0$ so that, since $\sum_{e\cdot e_1=0} \widetilde{\delta}(0,e)=-(\widetilde{\delta}(0,e_1)+\widetilde{\delta}(0,-e_1))$, $\beta_{U_\pm}$ becomes
\begin{align*}
\beta_{U_\pm} =& 
-\sum_{e\cdot e_1 =\pm 1}\E\left( \frac{P_{e,\omega}(H_0 > T_{U_\pm})}{P^2_{0,\bar{\omega}_0}(\widetilde{H}_0 > T_{U_\pm})}\right) \E((\widetilde{d}(0)\cdot e_1)\widetilde{\delta}(0,e) )\\
&+(\text{Var}_{\mu}(\omega(0,e_1))-\text{Var}_{\mu}(\omega(0,-e_1)))\E\left( \frac{P_{e_2,\omega}(H_0 > T_{U_\pm})}{P^2_{0,\bar{\omega}_0}(\widetilde{H}_0 > T_{U_\pm})}\right).
\end{align*} By virtue of (K3), the second term in this expression for $\beta_{U_\pm}$ vanishes and so we are left with
$$
\beta_{U_\pm}= (\text{Var}_{\mu}(\omega(0,e_1))-\text{Cov}_{\mu}(\omega(0,e_1),\omega(0,-e_1)))\E\left( \frac{P_{e_1,\omega}(H_0 \leq T_{U_\pm})-P_{-e_1,\omega}(H_0 \leq T_{U_\pm})}{P^2_{0,\bar{\omega}_0}(\widetilde{H}_0 > T_{U_\pm})}\right).
$$ In particular, it now follows from assumption (K4) and uniform ellipticity that
\begin{equation}\label{eq:beta1}
\beta_{U_+} \geq \frac{\rho}{4d}\sigma^2 \hspace{1cm}\text{ and }\hspace{1cm}\beta_{U_-} \leq -\frac{\rho}{4d}\sigma^2.
\end{equation} On the other hand, using \eqref{eq:gamma} and uniform ellipticity we obtain the bound
$$
|\gamma_{U_\pm}| \leq 8\epsilon\E\left(\left( 4d \sum_{e \in V} |\widetilde{\delta}(0,e)|\right)^2\right) \leq 8\epsilon
(4d)^2(2d) \E\|\widetilde{\delta}(0)\|_2^2 = 256\epsilon d^3\sigma^2,
$$ where for the second inequality above we have used the Cauchy-Schwarz inequality.  In particular, if $\epsilon \leq \frac{\rho}{2048d^4}$ then by our assumption (K5) we conclude that
$$
\E(g_{U_+}(0,0,\cdot)(\vec{d}\cdot e_1)(0,\cdot)) \geq \left(\frac{\rho}{4d}-256\epsilon d^3\right) \sigma^2 \geq \frac{\rho}{8d}\sigma^2 > 0
$$ and 
$$ 
\E(g_{U_-}(0,0,\cdot)(\vec{d}\cdot e_1)(0,\cdot))\leq -\left(\frac{\rho}{4d}-256\epsilon d^3\right) \sigma^2 + 4d\lambda \leq -\frac{\rho}{8d}\sigma^2 +4d\lambda< 0
$$ so that \eqref{eq:desig}, and ultimately \eqref{eq:drifco2}, now follows.

Finally, \eqref{eq:drifco} is a direct consequence of the fact that
$$
\E(g_{U_\pm}(0,0,\omega)\omega(0,e_i))=\E(g_{U_\pm}(0,\omega)\omega(0,-e_i))
$$ for all $i=2,\dots,d$ by assumption (K2), since the mapping $\omega \mapsto g_{U_\pm}(0,0,\omega)=(P_{0,\omega}(\widetilde{H}_0> T_{U_\pm}))^{-1}$ is invariant under the transformations $\omega \mapsto T_i\omega:=(\omega(x,T_i(e)))_{x \in \Z^d\,,\,e \in V}$ where, for $i=2,\dots,d$, $T_i$ is any rotation in $d \geq 3$ (or symmetry in $d=2$) preserving $\Z^d$ and $e_1$, and mapping $e_i$ to $-e_i$. This concludes the proof of Theorem \ref{theo:nokal}.



\begin{thebibliography}{99}
	
	\bibitem[1]{BDR14} Berger, N., Drewitz, A., Ram\'{\i}rez, A.F.: 
	Effective polynomial ballisticity condition for random walk in random environment.  Comm. Pure Appl. Math. 67, 1947-1973 (2014)
	
	\bibitem[2]{BGL13}
	Boucheron, S., Lugosi, G., Massart, P.: Concentration inequalities. Oxford University
	Press, Oxford (2013) 
	
	\bibitem[3]{BSZ03}  Bolthausen, E., Sznitman,  A.-S., Zeitouni, O.: Cut points and diffusive random walks in random environment. Ann. Inst. H. Poincar\'e Probab. Statist. 39, 
	527-555 (2003)
	
	\bibitem[4]{DR14}  Drewitz, A., Ram\'{\i}rez, A.F.:
	Selected topics in random walks in random environments. In:
	Topics in percolative and disordered systems, pp. 23-83. Springer, New York (2014)
	
	\bibitem[5]{FGL12} Fan, X., Grama, I., Liu., Q.: Hoeffding's inequality for supermartingales. Stochastic Process. Appl. 122 (10), 3545-3559 (2012)
	
	\bibitem[6]{GR18} Guerra, E., Ram\'\i rez, A. F.:  A proof of Sznitman's conjecture about ballistic RWRE. Accepted in Comm. Pure Appl. Math. arXiv id: 1809.02011.
	
	\bibitem[7]{K81} Kalikow, S.: Generalized random walk in a random environment. Ann. Probab. 9, 753-768 (1981)
	
	\bibitem[8]{LRSS} Laurent, C., Ram\'\i rez, A. F., Sabot, C.,  Saglietti, S.: Velocity estimates for symmetric random walks at low ballistic disorder. To appear in: Sidoravicius, V. (ed.), Sojourns in Probability and Statistical Physics, Springer. arXiv id: 1701.06308.
	
	\bibitem[9]{Sa04} Sabot, C.: Ballistic random walks in random environment at low disorder. Ann. Probab. 32 (4), 2996-3023 (2004)
	
	\bibitem[10]{Sz01}
	Sznitman, A.-S.: On a class of transient random walks in random environment. Ann. Probab. 29 (2), 724-765 (2001)
	
	\bibitem[11]{Sz02}
	Sznitman, A.-S.: An effective criterion for ballistic behavior of random walks in random environment. Probab. Theory Related Fields 122 (4). 509-544 (2002)
	
	\bibitem[12]{Sz03} Sznitman, A.-S.: On new examples of ballistic random walks in random environment.  Ann. Probab.  31 (1), 285-322 (2003)
	
	\bibitem[13]{Sz06} Sznitman, A.-S.: Random motions in random media. In: Mathematical statistical physics, pp. 219-242. Elsevier, Amsterdam (2006) 
	
	\bibitem[14]{Z12} Zeitouni, O.: Random walks in random environment. In: Computational complexity, pp. 2564-2577. Springer, New York (2012)	
	
\end{thebibliography}


\end{document}